\documentclass{amsart}    
\usepackage{amsmath,dsfont} 
\usepackage{paralist}
\usepackage{graphics} 
\usepackage{epsfig} 
\usepackage{graphicx}  \usepackage{epstopdf} 
\usepackage[colorlinks=true]{hyperref}
\hypersetup{urlcolor=blue, citecolor=blue} 

\newtheorem{Proposition}{Proposition}[section]

\newtheorem{Lemme}{Lemma}[section]
\newtheorem{Theoreme}{Theorem}[section]

\newtheorem{Corollaire}{Corollary}[section]
\newtheorem{Remarque}{Remark}

\def \R{\mathbb{R}}
\def \Rt{\mathbb{R}^3}

\def \ds{\displaystyle}

\usepackage{amssymb}

\title[] %
{On the fractional regularity  for an  elliptic nonlinear  singular drift equation}  
\author[ Oscar Jarr\'in]{}


\thanks{$^*$Corresponding author:  Oscar Jarr\'in}

\begin{document}
	\maketitle

	\centerline{\scshape Oscar Jarr\'in}
	\medskip
	{\footnotesize
		\centerline{Escuela de Ciencias Físicas y Matemáticas}
		\centerline{Universidad de Las Américas}
		\centerline{V\'ia a Nay\'on, C.P.170124, Quito, Ecuador}
	} 
	
\bigskip
\begin{abstract} We consider an elliptic equation with the fractional Laplacian operator $(-\Delta)^{\frac{\alpha}{2}}$ in the dissipative term, a singular integral operator ${\bf A}(\cdot)$  in the nonlinear term, and an external source $f$. The key example is the stationary (time-independent) counterpart of the surface quasi-geostrophic equation.

Under    suitable assumptions on  $f$ and natural assumptions on ${\bf A}(\cdot)$ in the setting of Sobolev spaces, our main result examines how the fractional power $\alpha$ propagates and optimally improves the regularity of weak $L^p$-solutions to this equation. 
\end{abstract}
\section{Introduction} This article examines the following fractional elliptic nonlinear singular drift equation:
\begin{equation}\label{Equation}
(-\Delta)^{\frac{\alpha}{2}} u + {\bf A}(u)\cdot \vec{\nabla}u =f, \qquad \text{div}({\bf A}(u))=0, \quad \alpha>0,
\end{equation}
where $u:\R^n \to \R$ denotes the solution, $f: \R^n \to \R$ represents a given external source, and the vector field ${\bf A}(u)$ is given by
\[ {\bf A}(u) = \big(A_1(u), \cdots, A_n(u)  \big). \]
Here,   each component $A_i$ is assumed to be a singular integral operator satisfying certain additional suitable properties, which are explicitly stated in expressions (\ref{Condition-A}) and (\ref{Condition-A-Sobolev}) below.  On the other hand, one the of the main features of this equation is  the fractional derivative operator $(-\Delta)^{\frac{\alpha}{2}}$, which can be  defined in the Fourier variable by the symbol $|\xi|^\alpha$. This operator  is motivated by several numerical and experimental works on fractional dissipative partial differential equations \cite{Meerschaert,Meerschaert1}.  

\medskip

The equation (\ref{Equation}) is primarily motivated by its evolution (time-dependent) counterpart:
\begin{equation}\label{Evolution-Equation-Intro}
\partial_t u + (-\Delta)^{\frac{\alpha}{2}} u + {\bf A}(u) \cdot \vec{\nabla} u = f, \qquad \text{div}({\bf A}(u)) = 0.
\end{equation}
Particularly, in the case $n=2$ and setting
\begin{equation}\label{Example}
A_1 = - \partial_2 (-\Delta)^{-\frac{1}{2}},  \qquad A_2= \partial_1 (-\Delta)^{-\frac{1}{2}}.
\end{equation}
as the  Riesz operators, we obtain the surface quasi-geostrophic (SQG) equation. This equation models the temperature or buoyancy of a strongly stratified fluid in a rapidly rotating regime and is a fundamental equation in geophysics and meteorology \cite{Pedlosky}. Note that with these specific nonlocal operators $A_i$, and since $n=2$, it directly follows that  $\text{div}({\bf A}(u))=0$. This fact makes this assumption reasonable for the more general case of equations (\ref{Equation}) and (\ref{Evolution-Equation-Intro}).

\medskip

The (SQG) equation and some of its variants have received considerable attention in the literature. For simplicity, we will only cite a few works below. One of the principal aims is to provide a better understanding of how the effects of the fractional dissipation term, given by $(-\Delta)^{\frac{\alpha}{2}}u$, along with the nonlinear effects of ${\bf A}(u)\cdot \vec{\nabla}u$, interact in the qualitative study of solutions. Results concerning global well-posedness and smoothing effects can be found, for instance, in \cite{Dong, Michael}.

\medskip

On the other hand, the long-term behavior of solutions is discussed in \cite{Wu-Liu}. There, the authors consider an external source $f$ belonging to the $L^p$-spaces with $p>2/\alpha$ and prove the existence of a global attractor, which is a compact set in the strong topology of the $L^2$-space.

\medskip

As pointed out in some previous works involving other evolution models \cite{Cortez, Chueshov, Jarrin2}, one of the main features of the global attractor is that, under certain conditions, it contains the stationary (time-independent) solutions of these models. In particular, for the evolution model (\ref{Evolution-Equation}), its stationary solutions are precisely characterized as the solutions of equation (\ref{Equation}).

\medskip

This fact motivates the study of solutions to equation (\ref{Equation}) in different settings. In the recent work \cite{Chamorro-Cortez}, for any fractional parameter $0<\alpha<2$, the authors established an initial result on the existence of weak solutions to equation (\ref{Equation}) in the setting of Sobolev spaces $\dot{H}^{\frac{\alpha}{2}}(\R^n)$. Subsequently, in the particular homogeneous case when $f\equiv 0$, they also provided a first regularity result for the solutions. We highlight that these regularity properties are driven by the fractional power $\alpha$. Specifically:
\begin{itemize}
    \item When $\frac{n+2}{3}<\alpha$, a  bootstrap argument shows that $u$ becomes smooth.
    \item In the case $1<\alpha \leq \frac{n+2}{3}$, it is proven that this argument also holds under the additional assumption that $u \in L^\infty(\R^n)$. Note that, in general, this assumption does not follow from the fact that $u \in \dot{H}^{\frac{\alpha}{2}}(\R^n)$. See \cite[Appendix B]{Chamorro-Cortez} for details.
    \item Finally, as also pointed out in \cite[Appendix B]{Chamorro-Cortez}, the case $0<\alpha\leq1$ remains open and is considerably more delicate due to the weaker effects of the fractional Laplacian operator.
\end{itemize}

 \medskip

In this context, the main purpose of this article is to introduce a different approach and new ideas to examine the regularity properties of solutions $u$ to equation (\ref{Equation}). Inspired by previous works on nonlinear elliptic equations \cite{Abramyan, Bhakta, Caffarelli, Dong2}, we consider the setting of $L^p$-spaces and, more generally, $L^{p,q}$-spaces. Specifically, following some of the ideas in \cite{Bjorland}, we first prove the existence of weak solutions to equation (\ref{Equation}) in these  spaces.

\medskip

Thereafter, assuming a given regularity for the source term $f$ and some boundedness properties for the operator ${\bf A}(\cdot)$ in the Sobolev spaces, our main result demonstrates how the fractional derivative operator $(-\Delta)^{\frac{\alpha}{2}}$ propagates and optimally improves this regularity for $u$. Here, inspired by the methodology developed in \cite{Jarrin}, the novel idea for studying the regularity of solutions is to utilize certain inherent properties of the parabolic equation (\ref{Evolution-Equation-Intro}).

\subsection*{Statement of the results} For completeness, we begin by introducing and recalling some of the main properties of Lorentz spaces. For a measurable function  $g: \R^n \to \R$  and for a parameter $\lambda \geq 0$, we define  the distribution function
\[ d_g(\lambda)= dx \Big( \left\{  x \in \Rt: \ | g(x)|>\lambda \right\} \Big), \]
where $dx$ denotes the Lebesgue measure. Then, the re-arrangement function $g^{*}$ is defined as follows:
\[ g^{*}(t)= \inf \{ \lambda \geq 0: \ d_g(\lambda) \leq t\}. \]
In this context,  for $1\leq p <+\infty$ and $1\leq q \leq +\infty$, the Lorentz space $L^{p,q}(\R^n)$ consists of measure functions $g: \R^n \to \R$ such that $\| g \|_{L^{p,q}}<+\infty$, where:
\begin{equation*}
\| g \|_{L^{p,q}}=
\begin{cases}\vspace{3mm}
\ds{\frac{q}{p} \left( \int_{0}^{+\infty} (t^{1/p} g^{*}(t))^{q} dt\right)^{1/q}}, \quad q<+\infty, \\
\ds{\sup_{t>0} \, t^{1/p} g^{*}(t)}, \quad q=+\infty.
\end{cases}
\end{equation*}
The quantity $\| g \|_{L^{p,q}}$ is often used as a norm, even thought it does not verify the triangle inequality. Nevertheless, there exists  an equivalent norm (strictly speaking) which makes these spaces into Banach spaces. Additionally, these spaces are homogeneous of degree $-\frac{n}{p}$ and for $1 \leq q_1 <  p < q_2\leq +\infty$ the following continuous  embedding  holds:
\[ L^{p,q_1}(\R^n) \subset L^{p}(\R^n)=L^{p,p}(\R^n) \subset L^{p,q_2}(\R^n). \]  
For $p=+\infty$, we  have the identity  $L^{\infty,\infty}(\R^n)=L^\infty(\R^n)$.

\medskip

As mentioned, our first result establishes the existence of weak $L^p$-solutions to  equation (\ref{Equation}). These solutions are obtained from small external sources $f$  by applying a contraction principle to the following fixed-point problem:
\begin{equation}\label{Fixed-Point-Intro}
 u = - (-\Delta)^{-\frac{\alpha}{2}} \text{div}\big(u\, {\bf A}(u)\big) + (-\Delta)^{-\frac{\alpha}{2}} f,   
\end{equation}
where, due to the assumption that $\text{div}({\bf A}(u))=0$, the nonlinear term in equation (\ref{Equation}) can be rewritten as ${\bf A}(u)\cdot \vec{\nabla} u= \text{div}\big(u\, {\bf A}(u)\big)$. 

\medskip

Handling  this nonlinear term  imposes some technical difficulties. On the one hand, note that the right-hand side of (\ref{Fixed-Point-Intro}) defines a nonlinear operator acting on $u$. The only information that $u\in L^p(\R^n)$ seems insufficient to prove that this  nonlinear operator is contractive.  To overcome this problem, we show that the action of the  operator $- (-\Delta)^{-\frac{\alpha}{2}}\text{div}\big(\cdot)$ is equivalently given by a convolution kernel $K_\alpha(\cdot)$. See Lemma \ref{Lem-Kernel} for further details. This convolution kernel is composed of  homogeneous functions that are not $L^p$-integrable but they possess good properties within the larger framework of the Lorentz spaces. In this context, we will see that the space $L^{\frac{n}{\alpha-1},\infty}(\R^n)$ naturally appears. Here, the condition $1<\frac{n}{\alpha-1}<+\infty$ directly imposes  a first  constraint on  the fractional power $\alpha$ as $1<\alpha<n+1$.

\medskip

On the other hand, the entire nonlinear term $- (-\Delta)^{-\frac{\alpha}{2}} \text{div}\big(u\, {\bf A}(u)\big)$ also imposes certain assumptions  on the nonlocal operator ${\bf A}(\cdot)$. Specifically, for  $1<p_1<+\infty$,  $1\leq q_1\leq +\infty$, and  a constant $C_{A}=C_A(n,p1,q_1)>0$ depending on $n,p_1$ and $q_1$, we will assume that:
\begin{equation}\label{Condition-A}
{\bf A}(0)=0, \quad  \text{and}\quad \|{\bf A}(u_1)-{\bf A}(u_2)\|_{L^{p_1,q_1}} \leq C_A\| u_1 - u_2 \|_{L^{p_1,q_1}}.
\end{equation}
Note that  these conditions are naturally satisfied by the key example of the operator ${\bf A}(\cdot)$ given in (\ref{Example}),  which makes them reasonable.

\medskip

In this setting, our first result is stated as follows:
\begin{Proposition}\label{Th-Existence} Assume that the operator ${\bf A}(\cdot)$ verifies (\ref{Condition-A}). Let $1<\alpha < \frac{n}{2}+1$ and  $\frac{n}{(n+1)-\alpha}<p<+\infty$. Assume that $(-\Delta)^{-\frac{\alpha}{2}}f \in L^{\frac{n}{\alpha-1},\infty}\cap L^p(\R^n)$.  Let $R>0$ be such that 
\begin{equation}\label{Control-source}
\max\left(  \left\| (-\Delta)^{-\frac{\alpha}{2}} f \right\|_{L^{\frac{n}{\alpha-1},\infty}}, \,  \| (-\Delta)^{-\frac{\alpha}{2}} f \|_{L^p} \right) \leq R.   
\end{equation} 

There exists a universal quantity $\eta_0=\eta_0(\alpha,n)>0$, depending only  on the fractional power $\alpha$ and the dimension $n$, such that if 
\begin{equation}\label{Condition-0-f}
  R\leq \eta_0,
\end{equation}
then  equation (\ref{Equation}) has a solution 
\[ u \in L^{\frac{n}{\alpha-1},\infty}\cap L^p(\R^n). \]
Moreover,  this solution is uniquely determined and satisfies  $\| u - (-\Delta)^{-\frac{\alpha}{2}}f  \|_{L^{\frac{n}{\alpha-1},\infty}} \leq  R$. 
\end{Proposition}

For a fixed $n$, recall that the $L^{\frac{n}{\alpha-1},\infty}$-space naturally imposes the condition $1<\alpha<n+1$. However, due to additional (technical) constraints explained in Remark \ref{Rmk-1} below, we must restrict our consideration to the case $1<\alpha < \frac{n}{2}+1$. To the best of our knowledge, the ideas used to prove this proposition are no longer valid in the complementary cases when $0<\alpha\leq 1$ or $\alpha\geq n/2+1$, which remain open questions.

\medskip

For fixed $n$ and  $\alpha$, note that the existence of $L^p$-solutions holds for parameters $p$ in the range $\frac{n}{(n+1)-\alpha}<p<+\infty$. Following some ideas from \cite{Bjorland} and \cite{Jarrin2}, we think that  this result could  be extended to the limit cases of the spaces $L^{\frac{n}{(n+1)-\alpha},\infty}(\R^n)$ and $L^\infty(\R^n)$. On the other hand, as before, the ideas used to prove this proposition are no longer valid in the complementary cases when $1\leq p \leq \frac{n}{(n+1)-\alpha}$, which  also remain open questions.

\medskip

Now, we state the main result concerning the regularity of weak $L^p$-solutions. It is worth noting that this result is independent of the previous one, as we primarily assume the existence of weak  $L^p$-solutions and focus on their optimal regularity gain. 

\medskip

This regularity gain is quantified by the fractional power $\alpha$, along with the assumed initial regularity of the external source $f$ and the boundedness properties of the operator ${\bf A}(\cdot)$ in equation (\ref{Equation}). Specifically, for any $\sigma\in \R$ and $1<r<+\infty$, we will assume that 
\begin{equation}\label{Condition-A-Sobolev}
 \| {\bf A}(u)\|_{\dot{W}^{\sigma,r}}\leq C\, \| u\|_{\dot{W}^{\sigma,r}}.
\end{equation}
As the previous assumption given in (\ref{Condition-A}), note that this properties is also naturally satisfied by the key example of this operator  given in (\ref{Example}).

\begin{Theoreme}\label{Th-Regularity} Assume that the operator ${\bf A}(\cdot)$ verifies (\ref{Condition-A}) and (\ref{Condition-A-Sobolev}). Let $\alpha>1$, $\max\left(1,\frac{n}{\alpha-1}\right)<p<+\infty$ and let $u \in   L^p(\R^n)\cap L^{1,\infty}(\R^n)$ be a weak solution of equation (\ref{Equation}) associated with the external source $f$. 

\medskip

\begin{enumerate}
    \item  For any $1<r<+\infty$ and for  a parameter $s\geq 0$, assume that $  f\in \dot{W}^{-1,r}\cap\dot{W}^{s,r}(\R^n)$.  Then, it follows that  $u \in \dot{W}^{s+\alpha,r}(\R^n)$. 

\medskip

\item Additionally, for $\varepsilon>0$ assume that  $ f \notin \dot{W}^{s+\varepsilon,r}(\R^n)$.  Then, it holds that $u\notin \dot{W}^{s+\alpha+\varepsilon,r}(\R^n)$.
\end{enumerate}
\end{Theoreme}

Let us briefly outline the proof strategy and make some remarks. We consider a weak $L^p\cap L^{1,\infty}$-solution to equation (\ref{Equation}). On the one hand, the \emph{a priori} information $u \in L^p(\R^n)$, with $\max\left(1,\frac{n}{\alpha-1}\right)<p<+\infty$, together with the parabolic framework of equation (\ref{Evolution-Equation-Intro}), allows us to establish that $u \in L^\infty(\R^n)$, which is a key element in the proof.

\medskip

On the other hand, the additional assumption that $u \in L^{1,\infty}(\R^n)$ is primarily technical. Specifically, a crucial tool in our proof is an estimate given by the fractional Leibniz rule, which is stated in Lemma \ref{Leibniz-rule} below. In this estimate, without the information that $u \in L^{1,\infty}(\R^n)$, we would be required to control the expression $\| A(u)\|_{L^\infty}$. However, in general, the singular integral operator ${\bf A}(\cdot)$ cannot be assumed to be a bounded operator on $L^{\infty}(\R^n)$ To circumvent this issue, we leverage the fact that $u \in L^{1,\infty}\cap L^\infty(\R^n)$, along with interpolation inequalities, to control the expression $\| A(u)\|_{L^r}$ for some $1<r<+\infty$. In this case, we can apply the natural boundedness assumption (\ref{Condition-A}) on the operator ${\bf A}(\cdot)$.

\medskip

Under this framework, given the initial regularity assumption on the external source in the first point above,  we apply a sharp bootstrap argument to the fixed-point equation (\ref{Fixed-Point-Intro}) to conclude that weak $L^{1,\infty}\cap L^p$-solutions to equation (\ref{Equation}) gain regularity and become $\dot{W}^{s+\alpha,r}$-solutions.

\medskip

For a given parameter $\alpha>1$,  this (expected) gain in fractional regularity arises from the effects of the fractional Laplacian operator $(-\Delta)^{\frac{\alpha}{2}}$ in equation (\ref{Equation}). It is worth emphasizing that this regularity gain is optimal due to the prescribed regularity $f\in \dot{W}^{s,r}(\R^n)$ and the assumption that $f\notin \dot{W}^{s+\varepsilon,r}(\R^n)$  for any $\varepsilon>0$.

\medskip

Theorem \ref{Th-Regularity} also yields the following corollary. Recall that for $k\in \mathbb{N}$ and $0<\sigma<1$, we denote by $\mathcal{C}^{k,\sigma}(\R^n)$ the  H\"older space of $\mathcal{C}^k$-functions whose derivatives of order $k$ are H\"older continuous functions with exponent $\sigma$.
 
\begin{Corollaire}\label{Corollary-Regularity} Under the same hypotheses as in Theorem \ref{Th-Regularity}, fix $r>n$ and  let $[s]$ denote the integer part of $s\geq 0$. Then, the weak solution $u \in L^{1,\infty} \cap  L^p(\R^n)$ to equation (\ref{Equation}) also satisfies $u \in \mathcal{C}^{[s], \sigma}(\R^n)$, where $\sigma=1-n/r$. In particular, in the homogeneous case when $f \equiv 0$, weak $L^{1,\infty}\cap L^p$-solutions to equation (\ref{Equation}) become $\mathcal{C}^{\infty}$-solutions.
\end{Corollaire}

\medskip

On the other hand, as pointed out in \cite{Chamorro-Cortez,Tang}, the study of smoothing effects in the case where $0<\alpha\leq 1$ remains a far from obvious open question. Roughly speaking, the fractional powers $\alpha>1$ govern the one-derivative terms in the nonlinear expression $\text{div}(u \, {\bf A}(u))$. Consequently, the bootstrap argument used to prove Theorem \ref{Th-Regularity} breaks down for values $0<\alpha\leq 1$. Refer to Remark \ref{Rmk-regularity-2} for further details.

\medskip

To illustrate this issue, we consider the following toy model for equation (\ref{Equation})
\begin{equation}\label{Toy-Model}
(-\Delta)^{\frac{\alpha}{2}}u + (-\Delta)^{\frac{\beta}{2}}(u^2) = f, \qquad \alpha, \beta > 0,
\end{equation}
where the nonlinear effects of the term $u \, {\bf A}(u)$ are represented by the simpler expression $u^2$, and the fractional derivative operator $(-\Delta)^{\frac{\beta}{2}}$ replaces the divergence operator. Note that when $\beta=1$, both equations (\ref{Equation}) and (\ref{Toy-Model}) exhibit the same regularity properties for their solutions.

\medskip

In Appendix \ref{Appendix}, we prove that a result similar to Theorem \ref{Th-Regularity} holds for equation (\ref{Toy-Model}) when $0<\beta< \alpha \leq 1$. Essentially, we observe that the weaker nonlinear effects in this equation allow us to consider small values for the parameter $\alpha$.

\medskip

{\bf Organization of the article}.  In Section \ref{Sec:Preliminaries} below, we recall some well-known tools that we will use in the sequel. Section \ref{Sec:Th-Existence} is devoted to proving Proposition \ref{Th-Existence}, while in Section \ref{Sec:Th-Regularity}, we provide a proof of the main result stated in Theorem \ref{Th-Regularity}, and we give a proof of Corollary \ref{Corollary-Regularity}. Finally, in Appendix \ref{Appendix}, we discuss the toy model (\ref{Toy-Model}).

\medskip

{\bf Notation}. In the rest of the article, we will use a generic $C>0$, which may change from one line to another. Additionally, we will use the notations $\mathcal{F}(g)$  or $\widehat{g}$ to denote the Fourier transform of $g$.

\section{Preliminaries}\label{Sec:Preliminaries}
\emph{Some inequalities in the setting of Lorentz spaces}. We start by recalling  the Young inequalities. For a proof, we refer to  \cite[Section $1.4.3$]{Chamorro}.
\begin{Lemme}\label{Lem-Young} Let $1<p,p_1,p_2<+\infty$ and $1\leq q,q_1, q_2 \leq+\infty$. There exists a generic constant $C>0$ such that  the following estimates hold:  
\begin{enumerate}
\item $\ds{\Vert g \ast h \Vert_{L^{p,q}} \leq C_1 \Vert g \Vert_{L^{p_1, q_1}}\, \Vert h \Vert_{L^{p_2, q_2}}}$, with $1+\frac{1}{p}=\frac{1}{p_1}+\frac{1}{p_2}$, $\frac{1}{q}\leq \frac{1}{q_1}+\frac{1}{q_2}$,  and   $C_1=C\,p\left( \frac{p_1}{p_1-1}\right)\left( \frac{p_2}{p_2-1}\right)$.  

\medskip
\item $\ds{\Vert g \ast h \Vert_{L^{p,q}} \leq C_2 \Vert g \Vert_{L^1} \, \Vert h \Vert_{L^{p,q}}}$, with $C_2= C\, \frac{p^2}{p-1}$. 

\medskip

\item $\ds{\Vert g \ast h \Vert_{L^\infty} \leq C_3 \Vert g\Vert_{L^{p,q}}\, \Vert h \Vert_{L^{p',q'}}}$, with  $1=\frac{1}{p}+\frac{1}{p'}$, $1 \leq \frac{1}{q}+\frac{1}{q'}$ and $C_3= C \left(\frac{p}{p-1} \right)\left(\frac{p'}{p' -1} \right)$.
\end{enumerate} 
\end{Lemme}

Now, we state the following interpolation result. A proof can be consulted in \cite[Corollary $2.4.2$ and $2.4.5$]{Chamorro}.
\begin{Lemme}\label{Lem-Interpolation} Let $1\leq p_1<p_2\leq +\infty$, $1\leq q,q_1,q_2\leq +\infty$ and $0<\theta<1$. Then, the following inequalities hold:
\begin{enumerate}
    \item $\ds{\| g \|_{L^{p,q}} \leq C\, \| g\|^{1-\theta}_{L^{p_1,q_1}}\, \| g \|^{\theta}_{L^{p_2,q_2}}}$, \vspace{2mm}

    \medskip
    
    \item $\ds{  \| g \|_{L^{p,q}} \leq C\, \| g \|^{1-\theta}_{L^{p_1}}\,\|g \|^{\theta}_{L^{p_2}}}$, \vspace{2mm}
\end{enumerate}
where $C>0$ is a constant and  $\frac{1}{p}=\frac{1-\theta}{p_1}+\frac{\theta}{p_2}$.  
\end{Lemme}
\emph{Estimates involving the fractional heat kernel}. Recall that the function $p_\alpha(t,x)$ denotes the fundamental solution of the fractional heat equation:
\begin{equation}\label{Fractional-Heat-Kernel}
\partial_t p_\alpha +(-\Delta)^{\frac{\alpha}{2}} p_\alpha = 0, \qquad p_\alpha(0,\cdot)=\delta_0, \quad \alpha>0,    
\end{equation}
where $\delta_0$ is the Dirac mass at the origin. 
\begin{Lemme}\label{Lem-Estim-Frac-Heat-Kernel} For any time $t>0$, the following estimates hold:
\begin{enumerate}
    \item $\ds{\| p_\alpha(t,\cdot)\|_{L^1}\leq C}$. 

    \medskip

    \item For any $1\leq q \leq +\infty$, we have $\ds{\left\| \vec{\nabla} p_\alpha(t,\cdot)\right\|_{L^q}\leq C\, t^{-\frac{1+n(1-1/q)}{\alpha}}}$.
 \end{enumerate}
\end{Lemme}
\begin{proof} These estimates directly follow from the identity $p_\alpha(t,x)=\frac{1}{t^{\frac{1}{\alpha}}}P_\alpha\left( \frac{x}{t^{\frac{1}{\alpha}}}\right)$, where the function $P_\alpha(\cdot)$ is defined as the inverse Fourier transform of $e^{-|\xi|^\alpha}$. Additionally, we have the following pointwise estimates:
\begin{equation*}
    0<P_\alpha(x)\leq \frac{C}{(1+|x|)^{n+\alpha}}, \qquad |\vec{\nabla}P_\alpha(x)| \leq \frac{C}{(1+|x|)^{n+\alpha+1}}.
\end{equation*}
See \cite[Chapter $3$]{Jacob} for a proof of these well-known facts.
\end{proof}
\begin{Lemme}\label{Lem-Lebesgue-Besov-Embedding} Let $1<p<+\infty$. For any time $t>0$ the following estimate  holds:
\begin{equation}\label{Estim-Lebesgue-Besov}
    t^{\frac{n}{\alpha p}}\| p_\alpha(t,\cdot)\ast g \|_{L^\infty}\leq C\,\| g \|_{L^p}.
\end{equation}
\end{Lemme}
\begin{proof} First,  recall that the continuous embedding holds $L^p(\R^n)\subset B^{-n/p}_{\infty,\infty}(\R^n)$. See, for instance, \cite[Page $171$]{PGLR1}. Here,   this non-homogeneous Besov space is characterized as the space of temperate distributions $g$ verifying $\ds{\| g\|_{B^{-n/p}_{\infty,\infty}}= \sup_{t>0}\, t^{\frac{n}{2p}}\| h_t \ast g \|_{L^\infty} <+\infty}$, where $h_t$ denotes the classical heat kernel. 

\medskip

Thereafter, from \cite[Page $9$]{PGLR2}, we have the following equivalence $\ds{\| g\|_{B^{-n/p}_{\infty,\infty}}\sim \sup_{t>0}\, t^{\frac{n}{\alpha p}}\| p_\alpha(t,\cdot)\ast g \|_{L^\infty}}$, from which the wished estimate follows. 
\end{proof}

\emph{Other useful estimates}. We will use the following fractional version of the  Leibniz rule. The proof of this estimate can be consulted in \cite{grafakos2014p} or \cite{naibo2019p}.
\begin{Lemme}\label{Leibniz-rule} Let $\alpha>0$, $1<p<+\infty$ and $1<p_0, p_1, q_0, q_1 \leq +\infty$. Then, there exist $C>0$ such that the following estimate holds:
\begin{equation*}
\left\| (-\Delta)^{\frac{\alpha}{2}}(gh) \right\|_{L^p} \leq C\, \left\| (-\Delta)^{\frac{\alpha}{2}} g \right\|_{L^{p_1}}\, \| h \|_{L^{p_2}}+C\, \| g \|_{L^{q_1}}\, \left\| (-\Delta)^{\frac{\alpha}{2}} h \right\|_{L^{q_2}},    
\end{equation*}
where $\frac{1}{p}=\frac{1}{p_1}+\frac{1}{p_2}=\frac{1}{q_1}+\frac{1}{q_2}$. 
\end{Lemme} 

Finally, we will  use the following result linking  Morrey spaces and the H\"older regularity of functions. For a proof, we refer to \cite[Proposition $3.4$]{Giga}. Recall that, for $1<p<+\infty$, the Morrey space $\dot{M}^{1,p}(\R^n)$ consists of locally finite Borel measures $d\mu$ such that 
\[ \sup_{x_0\in \R^n\, R>0} R^{\frac{n}{p}}\left( \int_{|x-x_0|<R} d |\mu|(x) \right) <+\infty.\]
Additionally, for $p<q\leq +\infty$,  the following embedding holds: $L^{p}(\R^n)\subset L^{p,q}(\R^n)\subset  \dot{M}^{1,p}(\R^n)$. 
\begin{Lemme}\label{Lem-Holder-Morrey} Let $g \in \mathcal{S}'(\R^n)$ be such that $\vec{\nabla} g \in \dot{M}^{1,p}(\R^n)$, with $n<p<+\infty$. Then, for $0<\sigma:=1-n/p<1$, it follows that $g \in \mathcal{C}^{0,\sigma}(\R^n)$. Specifically, there exists a constant $C>0$ such that for any $x,y \in \R^n$ the following inequality holds: $|g(x)-g(y)|\leq C\, \| \vec{\nabla} g\|_{\dot{M}^{1,p}}\, |x-y|^{\sigma}$.
\end{Lemme}

\section{Proof of Proposition \ref{Th-Existence}}\label{Sec:Th-Existence}
First, note that equation (\ref{Equation}) can be rewritten as the following (equivalent) fixed-point problem:
\begin{equation}\label{Fixed-point}
u = - (-\Delta)^{-\frac{\alpha}{2}} \text{div}\big(u\, {\bf A}(u)\big) + (-\Delta)^{-\frac{\alpha}{2}} f := T_\alpha(u).
\end{equation}
In the following technical lemma, we begin by studying the operator $\ds{- (-\Delta)^{-\frac{\alpha}{2}}\text{div}(\cdot)}$ in the nonlinear term above. Recall that this operator acts on vector fields. So, to state this lemma, we define the vector  field ${\bf v}=(\varphi_1, \cdots, \varphi_n)$, where $\varphi_j \in \mathcal{S}(\R^n)$ for   $j=1,\cdots, n$.

\begin{Lemme}\label{Lem-Kernel} For any $1<\alpha<n+1$, the action of operator $- (-\Delta)^{-\frac{\alpha}{2}}\text{div}(\cdot)$ is equivalently given by the product of convolution with the kernel $K_\alpha=(K_{\alpha,1},\cdots, K_{\alpha,n})$:
\begin{equation}\label{Identity-Kernel}
 - (-\Delta)^{-\frac{\alpha}{2}}\text{div}({\bf v})= K_\alpha \ast {\bf v}:=\sum_{j=1}^{n}K_{\alpha,j}\ast \varphi_j, \qquad \text{where} \quad    K_{\alpha,j}\in L^1_{loc}\cap L^{\frac{n}{(n+1)-\alpha},\infty} (\R^n).
\end{equation}
Moreover, there exists a constant  $C_{K}=C_K(\alpha,n)>0$, depending on the parameter $\alpha$ and the  dimension $n$, such that
\begin{equation}\label{Estim-Lorentz-Kernel}
    \| K_\alpha \|_{L^{\frac{n}{(n+1)-\alpha},\infty}}\leq C_{K}<+\infty. 
\end{equation}
\end{Lemme}
\begin{proof}
By a direct computation in the Fourier variable, we obtain
\begin{equation*}
    \mathcal{F}\Big(  - (-\Delta)^{-\frac{\alpha}{2}}\text{div}({\bf v}) \Big)(\xi)=\sum_{j=1}^{n} \frac{\textbf{i} \xi_j}{|\xi|^{\alpha}}\, \widehat{\varphi}_j (\xi),
\end{equation*}
where we define $\widehat{K}_{\alpha,j}(\xi)=\frac{\textbf{i} \xi_j}{|\xi|^{\alpha}}$. Since  $\widehat{K}_{\alpha,j}\in \mathcal{C}^{\infty}(\R^n \setminus \{0\})$ is a homogeneous function of degree $1-\alpha$, it follows that  $K_{\alpha,j}\in \mathcal{C}^{\infty}(\R^n \setminus \{0\})$  is a homogeneous function  of degree $\alpha-(n+1)$.  Since $\alpha>1$ it holds $-n<\alpha-(n+1)$, which implies 
that $K_{\alpha,i}\in L^1_{loc}(\R^n)$. Additionally, for any  $x\neq 0$, the following  pointwise estimate holds: 
\[ | K_\alpha (x)| \lesssim |x|^{\alpha-(n+1)}, \]
which yields the  estimate (\ref{Estim-Lorentz-Kernel}). 
\end{proof}

\medskip

Using this lemma, we first prove the existence of an $L^{\frac{n}{\alpha-1},\infty}$-solution to the fixed-point problem (\ref{Fixed-point}).

\begin{Lemme}\label{Prop-Existence-Lorentz} Let $1<\alpha<n/2+1$.   Assume that the operator ${\bf A}(\cdot)$ satisfies the boundedness condition (\ref{Condition-A}). Additionally, assume that $(-\Delta)^{-\frac{\alpha}{2}} f \in L^{\frac{n}{\alpha-1},\infty}(\R^n)$. On the other hand, let $R>0$ be the parameter introduced in (\ref{Control-source}). 

\medskip

There exists a universal  constant $\eta_1=\eta_1(\alpha,n)>0$,  also depending on $\alpha$ and $n$, such that if 
\begin{equation}\label{Control-1-f}
   R\leq \eta_1,
\end{equation}
 then the fixed-point problem (\ref{Fixed-point}) has a solution $u \in L^{\frac{n}{\alpha-1},\infty}(\R^n)$. Moreover,  this solution is uniquely determined and satisfies  $\| u - (-\Delta)^{-\frac{\alpha}{2}}f  \|_{L^{\frac{n}{\alpha-1},\infty}} \leq  R$.
\end{Lemme}
\begin{proof}
We will use Banach's contraction principle to construct a solution to equation (\ref{Fixed-point}). To achieve this, we consider the radius $R>0$ introduced above, which will be suitably fixed later as in expression (\ref{Control-1-f}). Therefore, we define the closed ball in the space $L^{\frac{n}{\alpha-1},\infty}(\R^n)$:
\begin{equation}\label{Ball}
 B_R:= \Big\{ v \in L^{\frac{n}{\alpha-1},\infty}(\R^n): \, \| v- (-\Delta)^{-\frac{\alpha}{2}}f \big\|_{L^{\frac{n}{\alpha-1},\infty}} \leq R \Big\}.   
\end{equation}

 We begin by proving that the nonlinear operator $T_\alpha(\cdot)$, defined in expression (\ref{Fixed-point}), maps the ball $B_R$ into itself. By identity (\ref{Identity-Kernel}), for any $v \in B_R$, we write
\begin{equation*}
    \| T_\alpha(v)-(-\Delta)^{-\frac{\alpha}{2}}f \|_{L^{\frac{n}{\alpha-1},\infty}} = \left\| -(-\Delta)^{\frac{\alpha}{2}}\text{div}(v\, {\bf A}(v))\right\|_{L^{\frac{n}{\alpha-1},\infty}}= \| K_\alpha \ast (v\, {\bf A}(v)) \|_{L^{\frac{n}{\alpha-1},\infty}}.
\end{equation*}
To control the last term,  in the first point of Lemma \ref{Lem-Young} we fix the parameters  $p=\frac{n}{\alpha-1}$, $p_1= \frac{n}{(n+1)-\alpha}$, $p_2=\frac{n}{2(\alpha-1)}$ and $q=q_1=q_2=+\infty$, to obtain 
\begin{equation}\label{Estim-Young-Kernel}
  \| K_\alpha \ast (v\, {\bf A}(v)) \|_{L^{\frac{n}{\alpha-1},\infty}} \leq C_1\, \| K_\alpha \|_{L^{\frac{n}{(n+1)-\alpha},\infty}}\, \| v\, {\bf A}(v)\|_{L^{\frac{n}{2(\alpha-1)},\infty}}. \end{equation}
 \begin{Remarque}\label{Rmk-1} The constraint $\alpha<n/2+1$ yields that $p_2=\frac{n}{2(\alpha-1)}>1$. 
 \end{Remarque}
Then, using the estimate (\ref{Estim-Lorentz-Kernel}), along with H\"older inequalities and the assumption on the operator ${\bf A}(\cdot)$ given in (\ref{Condition-A}), we can write  
 \begin{equation*}
\begin{split}
   &\,  C_1\, \| K_\alpha \|_{L^{\frac{n}{(n+1)-\alpha},\infty}}\, \| v\, {\bf A}(v)\|_{L^{\frac{n}{2(\alpha-1)},\infty}} \\
   \leq &\, C_1\, C_{K}\, \| v \|_{L^{\frac{n}{\alpha-1},\infty}}\, \| {\bf A}(v)\|_{L^{\frac{n}{\alpha-1},\infty}}\\
   \leq &\, C_1\, C_{K}\,C_A\, \| v \|^2_{L^{\frac{n}{\alpha-1},\infty}}\\
   = &\, C_{\alpha,n}\, \| v \|^2_{L^{\frac{n}{\alpha-1},\infty}}, 
  \end{split}
\end{equation*}
where the constant 
\begin{equation}\label{C-alpha-n}
 C_{\alpha,n}:=C_1\, C_{K}\,C_A>0,   
\end{equation} essentially depends on the fractional power  $\alpha$ and the dimension $n$. 

\medskip

Gathering these estimates, we obtain
\begin{equation}\label{Estim-NonLin-Lorentz} 
    \| K_\alpha \ast (v\, {\bf A}(v)) \|_{L^{\frac{n}{\alpha-1},\infty}} \leq C_{\alpha,n}\, \| v \|^2_{L^{\frac{\alpha}{n-1},\infty}}.
\end{equation}

Now, since $v\in B_R$, we can write
\begin{equation*}
   \begin{split} 
  C_{\alpha,n}\, \| v \|^2_{L^{\frac{\alpha}{n-1},\infty}} \leq &\, C_{\alpha,n}\, \left( \| v - (-\Delta)^{-\frac{\alpha}{2}} f \|^2_{L^{\frac{n}{\alpha-1},\infty}} + \| (-\Delta)^{-\frac{\alpha}{2}} f\|_{L^{\frac{n}{\alpha-1},\infty}} \right)^2\\
   \leq &\, C_{\alpha,n}\, \left( R+ \| (-\Delta)^{-\frac{\alpha}{2}} f\|_{L^{\frac{n}{\alpha-1},\infty}} \right)^2.
  \end{split}
  \end{equation*}
Additionally,  from estimate (\ref{Control-source}) we obtain 
\begin{equation*}
    C_{\alpha,n}\, \left( R+ \| (-\Delta)^{-\frac{\alpha}{2}} f\|_{L^{\frac{n}{\alpha-1},\infty}} \right)^2 \leq  C_{\alpha,n}\, 4R^2.
\end{equation*}
Then, for the constant $C_{\alpha,n}$ given in (\ref{C-alpha-n}), we define the constant 
\begin{equation}\label{eta1}
 \eta_1:= \frac{1}{8 C_{\alpha,n}},  
\end{equation}
and we fix the radius $R$ as in (\ref{Control-1-f}). This implies  the control
\begin{equation*}
  C_{\alpha,n}\, 4R^2 \leq  R,  
\end{equation*}
yielding the desired estimate $\ds{\| T_\alpha(v)- (-\Delta)^{-\frac{\alpha}{2}} f \|_{L^{\frac{n}{\alpha-1},\infty}}\leq R}$. 

\medskip

Now, we must verify that the operator $T_\alpha(\cdot):\, B_R \to B_R$ is contractive. For $v_1, v_2 \in B_R$, from the definition of $T_\alpha(\cdot)$ given in (\ref{Fixed-point}) and the identity (\ref{Identity-Kernel}),  using the previous estimates (\ref{Estim-Lorentz-Kernel}) and (\ref{Estim-Young-Kernel}), we write
\begin{equation*}
    \begin{split}
        \| T_\alpha(v_1)-T_\alpha(v_2) \|_{L^{\frac{n}{\alpha-1},\infty}}= &\,  \left\| K_\alpha \ast \big(v_1\, {\bf A}(v_1)-v_2\, {\bf A}(v_2)\big)\right\|_{L^{\frac{n}{\alpha-1},\infty}}\\
        \leq &\,C_1\, \| K_\alpha \|_{L^{\frac{n}{(n+1)-\alpha},\infty}}\, \| v_1\, {\bf A}(v_1)-v_2\, {\bf A}(v_2) \|_{L^{\frac{n}{2(\alpha-1)},\infty}}\\
        \leq &\, C_1\, C_{K}\, \| v_1\, {\bf A}(v_1)-v_2\, {\bf A}(v_2) \|_{L^{\frac{n}{2(\alpha-1)},\infty}}.
    \end{split}
\end{equation*}
Thereafter, using H\"older inequalities and  the assumptions given in (\ref{Condition-A}) and (\ref{Control-source}), we obtain
\begin{equation*}
\begin{split}
   &\,  \| v_1\, {\bf A}(v_1)-v_2\, {\bf A}(v_2) \|_{L^{\frac{n}{2(\alpha-1)},\infty}}  \\
   \leq &\, \| v_1({\bf A}(v_1)-{\bf A}(v_2))+ (v_1-v_2)\, {\bf A}(v_2)\|_{L^{\frac{n}{2(\alpha-1)},\infty}}\\
   \leq &\, C_A \left(\| v_1\|_{L^{\frac{n}{(\alpha-1)},\infty}}+\| v_2\|_{L^{\frac{n}{(\alpha-1)},\infty}}\right)\, \| v_1-v_2\|_{L^{\frac{n}{(\alpha-1)},\infty}}\\
   \leq &\, C_A \left( \| v_1-(-\Delta)^{-\frac{\alpha}{2}}f\|_{L^{\frac{n}{(\alpha-1)},\infty}}+ \| v_2-(-\Delta)^{-\frac{\alpha}{2}}f\|_{L^{\frac{n}{(\alpha-1)},\infty}}+2 \|(-\Delta)^{-\frac{\alpha}{2}}f \|_{L^{\frac{n}{\alpha-1},\infty}}\right) \times\\
   &\, \times \| v_1-v_2\|_{L^{\frac{n}{(\alpha-1)},\infty}}\\
   \leq &\, C_A \left( 2R +2 \|(-\Delta)^{-\frac{\alpha}{2}}f \|_{L^{\frac{n}{\alpha-1},\infty}}\right)\, \| v_1-v_2\|_{L^{\frac{n}{(\alpha-1)},\infty}}\\
   \leq &\, C_A \, 4R\, \| v_1-v_2\|_{L^{\frac{n}{(\alpha-1)},\infty}}.
   \end{split}
\end{equation*}

Gathering these estimates and using again the constant $C_{\alpha,n}$ introduced in (\ref{C-alpha-n}), we find that
\begin{equation*}
 \| T_\alpha(v_1)-T_\alpha(v_2) \|_{L^{\frac{n}{\alpha-1},\infty}} \leq C_{\alpha,n}\, 4R\, \| v_1-v_2\|_{L^{\frac{n}{(\alpha-1)},\infty}}.  
\end{equation*}
But, recalling that the radius $R$ was fixed in (\ref{Control-1-f}) where $\eta_1$ defined in (\ref{eta1}), we have   $\ds{C_{\alpha,n} 4R \leq \frac{1}{2}}$, 
which implies that $T_\alpha(\cdot): \, B_R \to B_R$ is contractive.  From now on, Proposition \ref{Prop-Existence-Lorentz} follows from well-known arguments. 
\end{proof}

Having constructed this $L^{\frac{n}{\alpha-1},\infty}$-solution to equation (\ref{Fixed-point}), and under additional assumptions on the external source $f$, we will prove that $u \in L^p(\R^n)$.
 
\begin{Lemme}\label{Prop-Existence-Lp} Under the same hypothesis of Proposition \ref{Prop-Existence-Lorentz}, for $\frac{n}{(n+1)-\alpha}<p<+\infty$, assume that $(-\Delta)^{-\frac{\alpha}{2}}f \in L^p(\R^n)$. On the other hand, let $R>0$ be the parameter introduced in (\ref{Control-source}) and fixed as in (\ref{Control-1-f}).

\medskip

There exists a universal constant $\eta_2=\eta_2(\alpha,n)>0$, depending only on $\alpha$ and $n$, such that if $R$ additionally verifies
\begin{equation}\label{Control-R-2}
 R\leq  \eta_2,
\end{equation}
then  it holds that $u \in L^p(\R^n)$.
\end{Lemme}
\begin{proof} Recall that the solution $u$ to equation (\ref{Fixed-point}),  given by Proposition \ref{Prop-Existence-Lorentz}, is obtained as the limit in the strong topology of the space $L^{\frac{n}{\alpha-1},\infty}(\R^n)$:
\begin{equation*}
    u=\lim_{n\to +\infty} u_n,  \qquad u_n \in B_R,
\end{equation*}
where the closed  ball $B_R$ was defined in expression (\ref{Ball}) and,  for any $n\in \mathbb{N}$, the term $u_n$ is defined by 
following iterative formula:
\begin{equation}\label{Iterative-Formula}
\begin{cases}\vspace{2mm}
u_n := T_\alpha(u_{n-1}), \quad n\geq 1, \\
u_0:= (-\Delta)^{-\frac{\alpha}{2}}f. 
\end{cases}
\end{equation}
In particular, from the control assumed in (\ref{Control-source}) the following uniform bound holds: 
\begin{equation}\label{Uniform-Bound-Sequence}
 \| u_{n}\|_{L^{\frac{n}{\alpha-1},\infty}} \leq R +\|(-\Delta)^{-\frac{\alpha}{2}}f \|_{L^{\frac{n}{\alpha-1},\infty}} \leq 2R,   
\end{equation}
which we will use below. 

\medskip

In this context, we will prove that the additional control  (\ref{Control-R-2})  imply  that the sequence $(u_n)_{n\in \mathbb{N}}$ is uniformly bounded in the strong topology of the $L^p$-space.

\medskip

For $n=0$, it directly follows that $u_0 \in L^p(\R^n)$. Then, for $n\geq 1$, using the definition of the operator $T_\alpha(\cdot)$, given in (\ref{Fixed-point}), together with the identity (\ref{Identity-Kernel}) and the assumption (\ref{Control-source}),  we write
\begin{equation*}
\begin{split}
    \| u_n \|_{L^p} = &\, \| T_\alpha(u_{n-1})\|_{L^p} \leq \left\| (-\Delta)^{-\frac{\alpha}{2}}\text{div}\big(u_{n-1}\,{\bf A}(u_{n-1})\big)\right\|_{L^p}+\| (-\Delta)^{-\frac{\alpha}{2}}f \|_{L^p} \\
    \leq &\, \left\| K_\alpha \ast \big(u_{n-1}\,{\bf A}(u_{n-1})\big) \right\|_{L^p}+ R.
    \end{split}
\end{equation*}
Here, we must control the term $\left\| K_\alpha \ast \big(u_{n-1}\,{\bf A}(u_{n-1})\big) \right\|_{L^p}$. In the first point of Lemma \ref{Lem-Young}, we fix the parameters $p_1=\frac{n}{(n+1)-\alpha}$, $p_2=\frac{np}{n+p(\alpha-1)}$, $q=p$,  $q_1=+\infty$ and $q_2=p$. Additionally,  using the estimate (\ref{Estim-Lorentz-Kernel}), we obtain
\begin{equation*}
\begin{split}
    \left\| K_\alpha \ast \big(u_{n-1}\,{\bf A}(u_{n-1})\big) \right\|_{L^p} \leq &\,  C_1\, \| K_\alpha \|_{L^{\frac{n}{(n+1)-\alpha},\infty}}\, \| u_{n-1}\,{\bf A}(u_{n-1})\|_{L^{\frac{np}{n+p(\alpha-1)},\infty}}\\
    \leq &\, C_1\, C_{K}\, \| u_{n-1}\,{\bf A}(u_{n-1})\|_{L^{\frac{np}{n+p(\alpha-1)},\infty}}.
 \end{split}
\end{equation*}
\begin{Remarque}\label{Rmk-2} Since  $p_1=\frac{n}{(n+1)-\alpha}$ and  $p_2=\frac{np}{n+p(\alpha-1)}$, the constant $C_1$ above is explicitly computed as the following function of the parameter $p$:
\begin{equation}\label{Expression-C1} 
    C_1= \, C\,p \left( \frac{p_1}{p_1 -1} \right)\, \left(\frac{p_2}{p_2-1}\right)=\, C\, p \left( \frac{n}{\alpha-1} \right)\left( \frac{np}{p\big((n+1)-\alpha\big)-n} \right):=C_1(p).
\end{equation}
Consequently, the  condition $0<C_1(p)<+\infty$ implies the constraint on $p$ given in the statement of Proposition \ref{Prop-Existence-Lp}:
\begin{equation*}
\frac{n}{(n+1)-\alpha} < p < +\infty.    
\end{equation*}
\end{Remarque}

Thereafter, using H\"older inequalities and the assumption (\ref{Condition-A}) on the operator ${\bf A}(\cdot)$, we write
\begin{equation*}
\begin{split}
 &\,    C_1(p)\, C_{K}\, \| u_{n-1}\,{\bf A}(u_{n-1})\|_{L^{\frac{np}{n+p(\alpha-1)},\infty}} \\
 \leq &\,  C_1(p)\, C_{K}\, \|u_{n-1}\|_{L^p}\, \| {\bf A}(u_{n-1}) \|_{L^{\frac{n}{\alpha-1},\infty}}\\
    \leq &\,  C_1(p)\, C_{K}\, C_A\,  \|u_{n-1}\|_{L^p}\, \| u_{n-1} \|_{L^{\frac{n}{\alpha-1},\infty}}.
  \end{split}
\end{equation*}
Finally, we use the uniform control given in (\ref{Uniform-Bound-Sequence}) to obtain 
\begin{equation*}
  \left\| K_\alpha \ast \big(u_{n-1}\,{\bf A}(u_{n-1})\big) \right\|_{L^p} \leq \,  C_1(p)\, C_{K}\, C_A\, 2R \, \| u_{n-1}\|_{L^p}. 
\end{equation*}

Thus, for any $n\geq 1$, we obtain the following iterative estimate:
\begin{equation}\label{Iterative}
    \| u_{n}\|_{L^p} \leq   C_1(p)\, C_{K}\, C_A\, 2R\, \| u_{n-1}\|_{L^p}+ R.
\end{equation}

At this point, we need to control the expression $\ds{C_1(p)\, C_{K}\, C_A\, 2R}$. Specifically, we will need  the new constraint   on the radius $R$: 
\begin{equation}\label{Control-R-tech}
 C_1(p)\, C_{K}\, C_A\, 2R \leq \frac{1}{2}. 
\end{equation}
However, returning to Remark \ref{Rmk-2}, recall  that for any $\frac{n}{(n+1)-\alpha} < p < +\infty$, the expression  $C_1(p)$ is  a function  on  $p$.  Consequently,  one cannot directly assume the inequality $R \leq \frac{1}{4\,C_1(p)\, C_{K}\, C_A}$. To overcome this technical difficulty, we will consider the following cases for the parameter  $p$. 

\medskip

Since $1<\alpha < n/2+1$, it follows that $\frac{n}{(n+1)-\alpha}<2<\frac{n}{\alpha-1}$. Therefore, we can split the interval 
\[ \left( \frac{n}{(n+1)-\alpha}, +\infty \right)= \left(\frac{n}{(n+1)-\alpha}, 2 \right)\cup \left[2, \frac{3n}{\alpha-1} \right]\cup \left( \frac{3n}{\alpha-1}, +\infty \right):= I_1 \cup I_2\cup I_3. \]
Additionally, we begin by considering the interval $I_2$, since further computations on the other intervals are based on those in this  compact interval.

\begin{itemize}
    \item The case $p \in I_2$. Since $I_2$ is compact and the expression $C_1(p)$, given in (\ref{Expression-C1}), is continuous in the variable $p$, let $\ds{M_\alpha:= \sup_{p\in I_2}C_1(p)<+\infty}$. Then, we can assume the following bound as stated in (\ref{Control-R-2}):
    \begin{equation*}
         R \leq  \frac{1}{4\,M_\alpha\, C_{K}\, C_A}:=\eta_2,
    \end{equation*}
    from which the inequality (\ref{Control-R-tech}) holds for any $p\in I_2$. Having this inequality, returning to the   estimate  (\ref{Iterative}), we obtain 
    \begin{equation*}
      \| u_{n}\|_{L^p} \leq \frac{1}{2}\, \| u_{n-1}\|_{L^p}+R.      
    \end{equation*}

By iterating this estimate, it follows that 
\begin{equation*}
    \| u_{n}\|_{L^p} \leq \left( \sum_{j=0}^{+\infty}\frac{1}{2^j} \right)\, R.
\end{equation*}
Consequently, the sequence $(u_n)_{n\in \mathbb{N}}$ defined in (\ref{Iterative-Formula}) is uniformly bounded in the space $L^p(\R^n)$, which implies  that $u\in L^p(\R^n)$ for $p\in I_2$. 

\medskip

\item The case $p\in I_1$. Recall that we have assumed $(-\Delta)^{-\frac{\alpha}{2}}f \in L^p(\R^n)$. Then, by the assumption $(-\Delta)^{-\frac{\alpha}{2}} f \in L^{\frac{n}{\alpha-1},\infty}(\R^n)$ and since $p<2<\frac{n}{\alpha-1}$, the first point of Lemma \ref{Lem-Interpolation}  yields that $(-\Delta)^{-\frac{\alpha}{2}}f \in L^2(\R^n)$. Additionally, by the assumption (\ref{Control-source}) we also have $\| (-\Delta)^{-\frac{\alpha}{2}}f \|_{L^2}\leq R$.  Thus, by the previous case, it also holds that $u \in L^2(\R^n)$.

\medskip

With this information at our disposal, we will prove that $\ds{-(-\Delta)^{-\frac{\alpha}{2}}\text{div}(u \, {\bf A}(u))\in L^p(\R^n)}$. In fact, first note that by the assumption (\ref{Condition-A}) on the operator ${\bf A}(\cdot)$, it follows that ${\bf A}(u) \in L^2(\R^n)$, and  hence $u \, {\bf A}(u) \in L^1(\R^n)$. Additionally, using the identity (\ref{Identity-Kernel}) and the fact that $K_\alpha \in L^{\frac{n}{(n+1)-\alpha},\infty}(\R^n)$, we can apply  the second point of Lemma \ref{Lem-Young}, which gives us that   $\ds{K_\alpha\ast (u \, {\bf A}(u))\in L^{\frac{n}{(n+1)-\alpha},\infty}(\R^n)}$. 

\medskip

On the other hand, recall that by the estimate (\ref{Estim-NonLin-Lorentz}), we also have $K_\alpha\ast (u \, {\bf A}(u))\in L^{\frac{n}{\alpha-1},\infty}(\R^n)$. Then, by the first point of Lemma \ref{Lem-Interpolation}, for $p\in I_1$ we obtain $K_\alpha\ast (u \, {\bf A}(u)) \in L^p(\R^n)$. Finally, since $u$ verifies the equation (\ref{Fixed-point}), we conclude that   $u \in L^p(\R^n)$.  

\medskip

\item The case $p \in I_3$.  As before, we start by recalling the assumption that $(-\Delta)^{-\frac{\alpha}{2}}f \in L^p(\R^n)$, which, together by the information $(-\Delta)^{-\frac{\alpha}{2}}f \in L^{\frac{n}{\alpha-1},\infty}(\R^n)$ and the first point of Lemma \ref{Lem-Interpolation}, imply that $(-\Delta)^{-\frac{\alpha}{2}}f \in L^{\frac{n+1}{\alpha-1}}(\R^n)\cap L^{\frac{2n}{\alpha-1}}(\R^n)$. Since $\frac{n+1}{\alpha-1}, \frac{2n}{\alpha-1}\in I_2$ and using again the assumption (\ref{Control-source}),  by the first case it follows that $u \in  L^{\frac{n+1}{\alpha-1}}(\R^n)\cap L^{\frac{2n}{\alpha-1}}(\R^n)$. Therefore, applying the first point of Lemma \ref{Lem-Interpolation}, we find that $u\in L^{\frac{2n}{\alpha-1},2}(\R^n)$. 

\medskip

With this information, we now  prove that $\ds{-(-\Delta)^{-\frac{\alpha}{2}}\text{div}(u\,{\bf A}(u))\in L^\infty(\R^n)}$. Indeed, since $u\in L^{\frac{2n}{\alpha-1},2}(\R^n)$, by the assumption (\ref{Condition-A}) and H\"older inequalities we obtain $u \, {\bf A}(u)\in L^{\frac{n}{\alpha-1},1}(\R^n)$. Then, using the identity (\ref{Identity-Kernel}) and the well-known fact that $K_\alpha \in L^{\frac{n}{(n+1)-\alpha},\infty}(\R^n)$, the third point of Lemma \ref{Lem-Young} yields that $\ds{-(-\Delta)^{-\frac{\alpha}{2}}\text{div}(u\,{\bf A}(u))\in L^\infty(\R^n)}$. 

\medskip

Finally, since it also holds that $\ds{-(-\Delta)^{-\frac{\alpha}{2}} \text{div}(u \, {\bf A}(u)) \in L^{\frac{n}{\alpha-1}, \infty}(\R^n)}$, by  the first point of Lemma \ref{Lem-Interpolation} and using the identity (\ref{Fixed-point}), we obtain $u \in L^p(\R^n)$. 
 \end{itemize}
\end{proof}

Now, Proposition \ref{Th-Existence}  directly follows from Lemmas \ref{Prop-Existence-Lorentz} and \ref{Prop-Existence-Lp}. From  the constants $\eta_1$ and $\eta_2$ given in (\ref{eta1}) and (\ref{Control-R-2}), respectively,  we define the constant $\eta_0:= \min(\eta_1,\eta_2)$ and fix  $R\leq \eta_0$.

\section{Gain of fractional regularity}\label{Sec:Th-Regularity}
\subsection{Proof of Theorem \ref{Th-Regularity}}
For clarity, we divide the proof into the following steps.
\subsection*{Step 1: The parabolic setting} Here we study the evolution problem for equation  (\ref{Equation}):
\begin{equation}\label{Evolution-Equation}
\begin{cases}\vspace{2mm}
    \partial_t v  +(-\Delta)^{\frac{\alpha}{2}} v + \text{div}\big(v\, {\bf A}(v)\big)=g,\\
    v(0,\cdot)=v_0,
  \end{cases}  
\end{equation}
where $v_0$ denotes the initial datum and $g$ denotes a time-independent external source. For a time $0<T<+\infty$, we denote $\mathcal{C}_{*}([0,T],L^p(\R^n))$ the space of bounded and weak-$*$ continuous functions from $[0,T]$ with values in $L^p(\R^n)$. In this framework, we have the following technical result:
\begin{Proposition}\label{Prop-Evolution-Lp} Assume that the operator ${\bf A}(\cdot)$ verifies (\ref{Condition-A}). Let $\alpha>1$  and $\max\left(1, \frac{n}{\alpha-1}\right)<p<+\infty$. Assume that  $v_0,\, g  \in L^p(\R^n)$.

\medskip

Then, there exists a time $T_0>0$, depending on the initial datum $v_0$ and the external source $g$, and there exists a function
\begin{equation*}
    v \in \mathcal{C}_{*}([0,T_0], L^p(\R^n)),
\end{equation*}
which is the unique solution to  equation (\ref{Evolution-Equation}). Additionally, this solution verifies the following estimate:
\begin{equation}\label{Estim-Solution-Besov}
    \sup_{0 < t \leq T_0} t^{\frac{n}{\alpha p}}\| v(t,\cdot)\|_{L^\infty}<+\infty. 
\end{equation}
\end{Proposition}
\begin{proof} First, note that  equation (\ref{Evolution-Equation}) can be rewritten as the following  (equivalent) mild formulation:
\begin{equation}\label{Mild}
    v(t,\cdot)=p_\alpha(t,\cdot)\ast v_0 +\int_{0}^{t}p_\alpha(t-\tau,\cdot)\ast g\, d\tau - \int_{0}^{t}p_\alpha(t-\tau,\cdot)\ast \text{div}\big(v\, {\bf A}(v)\big)(\tau,\cdot)\, d\tau:=\mathcal{T}_\alpha(v), 
\end{equation}
where the function $p_\alpha(t,x)$ was introduced in (\ref{Fractional-Heat-Kernel}). Note that the right-hand side of the expression above defines the nonlinear operator $\mathcal{T}_\alpha(\cdot)$. For a time $0<T\leq 1$, which we will fix small enough below, we will prove that $\mathcal{T}_\alpha(\cdot)$ satisfies the hypotheses of the Banach's contraction principle in a closed ball of following  the Banach space:
\begin{equation*}
    E_T:=\left\{ w \in \mathcal{C}_{*}([0,T], L^p(\R^n)) : \, \sup_{0< t \leq T} \, t^{\frac{n}{\alpha p}}\| w(t,\cdot)\|_{L^\infty}<+\infty \right\},
\end{equation*}
endowed with the norm
\begin{equation*}
    \| w \|_{E_T}:=\sup_{0\leq t \leq T}\| w(t,\cdot)\|_{L^p}+\sup_{0< t \leq T} \, t^{\frac{n}{\alpha p}}\| w(t,\cdot)\|_{L^\infty}.
\end{equation*}

We will divide each step of the proof into the following technical lemmas. In   expression (\ref{Mild}), we begin by defining: 
\begin{equation}\label{Mild-Data}
U_0:= p_\alpha(t,\cdot)\ast v_0 +\int_{0}^{t}p_\alpha(t-\tau,\cdot)\ast g\, d\tau.   
\end{equation}
\begin{Lemme} Let $v_0, \, g \in L^p(\R^n)$. Then, it follows that $U_0 \in E_T$ and the following estimate holds:
\begin{equation}\label{Estim-Data-Terms}
    \| U_0 \|_{E_T} \leq C\, \big( \| v_0 \|_{L^p}+ \|g\|_{L^p} \big).
\end{equation}
\end{Lemme}
\begin{proof} For the initial datum term given in expression (\ref{Mild-Data}), since $v_0 \in L^p(\R^n)$, it follows  from the first point of Lemma \ref{Lem-Estim-Frac-Heat-Kernel} and Young inequalities that $\ds{\sup_{0\leq t \leq T}\|p_\alpha(t,\cdot)\ast v_0\|_{L^p} \leq C\,|v_0\|_{L^p}}$,  hence $p_\alpha(t,\cdot)\ast v_0 \in \mathcal{C}_{*}([0,T],L^p(\R^n))$. Additionally, using  Lemma \ref{Lem-Lebesgue-Besov-Embedding}, we find that  
\begin{equation}\label{Estim-Initial-Datum}
\sup_{0<t\leq T}\, t^{\frac{n}{\alpha p}}\| p_\alpha(t,\cdot)\ast v_0\|_{L^\infty}\leq C\, \|v_0\|_{L^p}.    
\end{equation}

We now study the external source term in (\ref{Mild-Data}).  Since the external source $g$ is a time-independent function, and since $0<T\leq 1$,  we directly obtain
\begin{equation*}
  \sup_{0\leq t \leq T}  \left\| \int_{0}^{t}p_\alpha(t-\tau,\cdot)\ast g \, d\tau \right\|_{L^p} \leq C\, T\, \| g \|_{L^p}\leq C\, \| g\|_{L^p}.
\end{equation*}
Additionally, from  estimate (\ref{Estim-Lebesgue-Besov}), and since $p>\frac{n}{\alpha-1}>\frac{n}{\alpha}$ (hence we obtain that $1-\frac{n}{\alpha p}>0$), we can write
\begin{equation*}
\begin{split}
   &\,  \sup_{0\leq t \leq T}\, t^{\frac{n}{\alpha p}}\left\| \int_{0}^{t} p_\alpha(t-\tau,\cdot)\ast g\, d\tau\right\|_{L^\infty} \leq  \sup_{0\leq t \leq T}\, t^{\frac{n}{\alpha p}} \int_{0}^{t} \| p_\alpha(t-\tau,\cdot)\ast g\|_{L^\infty}\, d\tau\\
   =&\, \sup_{0\leq t \leq T}\, t^{\frac{n}{\alpha p}} \int_{0}^{t} \frac{(t-\tau)^{\frac{n}{\alpha p}}}{(t-\tau)^{\frac{n}{\alpha p}}} \| p_\alpha(t-\tau,\cdot)\ast g\|_{L^\infty}\, d\tau\leq \sup_{0\leq t \leq T}\, t^{\frac{n}{\alpha p}}\, C\, \| g\|_{L^p}\left( \int_{0}^{t} \frac{d \tau}{(t-\tau)^{\frac{n}{\alpha p}}} \right)\\
   \leq &\, C\,T\, \| g\|_{L^p} \leq C\, \| g \|_{L^p}.
    \end{split}
\end{equation*}
Gathering these estimates, we find that $\ds{\int_{0}^{t}p_\alpha(t-\tau,\cdot)\ast g\, d\tau \in E_T}$, where the following estimate holds:
\begin{equation}\label{Estim-External-Source}
    \left\| \int_{0}^{t}p_\alpha(t-\tau,\cdot)\ast g\, d\tau \right\|_{E_T} \leq C\,  \| g \|_{L^p}.
\end{equation}

Finally, the wished estimate (\ref{Estim-Data-Terms}) follows from estimates (\ref{Estim-Initial-Datum}) and (\ref{Estim-External-Source}). 
\end{proof}

Now, for a fixed radius $R>0$, we define the ball  $\ds{B_R := \big\{  w \in E_T : \, \| w-U_0 \|_{E_T} \leq R  \big\}}$.
\begin{Lemme}\label{Lem-Parabolic-1} Assume that the operator ${\bf A}(\cdot)$ verifies (\ref{Condition-A}). Fix the time $0<T\leq 1$ sufficiently small as in expression (\ref{Control-Time-1}) below. Then, the operator $\mathcal{T}_{\alpha}(\cdot)$, defined in (\ref{Mild}), maps the ball $B_R$ into itself. 
\end{Lemme}
\begin{proof} Let $w \in B_R$. Since $\ds{\mathcal{T}_\alpha(w)-U_0=\int_{0}^{t}p_\alpha(t-\tau,\cdot)\ast \text{div}\big( w\,{\bf A}(w)\big)(\tau,\cdot)\, d\tau}$, we will estimate this expression term by term in the norm  $\| \cdot \|_{E_T}$. 

\medskip
    
 Using the second estimate in Lemma \ref{Lem-Estim-Frac-Heat-Kernel} (with  $q=1$), and using the assumption (\ref{Condition-A}), we write
\begin{equation*}
    \begin{split}
     &\,    \sup_{0\leq t \leq T} \left\| \int_{0}^{t} p_\alpha(t-\tau,\cdot)\ast \text{div}\big( w\, {\bf A}(w)\big)(\tau,\cdot)\, d\tau  \right\|_{L^p}\\
     \leq &\, \sup_{0\leq t \leq T} \, \int_{0}^{t} \| \vec{\nabla}p_\alpha (t-\tau,\cdot)\|_{L^1}\, \| w \,{\bf A}(w)(\tau,\cdot)\|_{L^p}\, d\tau\\
     \leq &\, \sup_{0\leq t \leq T} \, C\, \int_{0}^{t} \frac{1}{(t-\tau)^{\frac{1}{\alpha}}}\, \| w(\tau,\cdot)\|_{L^\infty}\, \| {\bf A}(w)(\tau,\cdot)\|_{L^p}\, d \tau\\
       \leq &\, \sup_{0\leq t \leq T} \, C\, \int_{0}^{t} \frac{1}{(t-\tau)^{\frac{1}{\alpha}}}\, \frac{\tau^{\frac{n}{\alpha p}}}{\tau^{\frac{n}{\alpha p}}} \| w(\tau,\cdot)\|_{L^\infty}\, \| w(\tau,\cdot)\|_{L^p}\, d \tau\\
       \leq &\, \sup_{0\leq t \leq T} \, C\,\left( \int_{0}^{t} \frac{d\tau}{(t-\tau)^{\frac{1}{\alpha}}\, \tau^{\frac{n}{\alpha p}}}\right) \left( \sup_{0\leq \tau\leq T}  \tau^{\frac{n}{\alpha p}}\| w(\tau,\cdot)\|_{L^\infty}\right)\, \left( \sup_{0\leq \tau\leq T}\| w(\tau,\cdot)\|_{L^p}\right)\\
       \leq &\, C\, T^{-\frac{1}{\alpha}-\frac{n}{\alpha p}+1}\, \| w\|^2_{E_T}.
    \end{split}
\end{equation*}
In the last line, note that the integral converges due to the conditions $\alpha>1$ and $p>\frac{n}{\alpha-1}>\frac{n}{\alpha}$. Additionally, from the inequality $p>\frac{n}{\alpha-1}$, it follows that $-\frac{1}{\alpha}-\frac{n}{\alpha p}+1>0$. 

\medskip

Now, for the parameter $q$ such that $1=1/q+1/p$, applying Young inequalities with the relation $1+1/\infty=1/q+1/p$, and using again the second point of Lemma \ref{Lem-Estim-Frac-Heat-Kernel}, we obtain
\begin{equation*}
    \begin{split}
       &\,  \sup_{0\leq t \leq T}\, t^{\frac{n}{\alpha p}}\, \left\| \int_{0}^{t} p_\alpha(t-\tau,\cdot)\ast \text{div}(w\, {\bf A}(w))(\tau,\cdot)\, d \tau \right\|_{L^\infty}\\
       \leq &\, \sup_{0\leq t \leq T}\, t^{\frac{n}{\alpha p}} \, \int_{0}^{t} \| \vec{\nabla}p_\alpha(t-\tau,\cdot)\|_{L^q}\, \| w\, {\bf A}(w) (\tau,\cdot)\|_{L^p}\, d\tau\\
       \leq &\, \sup_{0\leq t \leq T}\, t^{\frac{n}{\alpha p}}\, C\, \int_{0}^{t} \frac{1}{(t-\tau)^{\frac{1+n/p}{\alpha}}}\, \|w(\tau,\cdot)\|_{L^\infty}\, \| {\bf A}(w)(\tau,\cdot)\|_{L^p}\, d\tau\\
       \leq &\, \sup_{0\leq t \leq T}\, t^{\frac{n}{\alpha p}}\, C\, \left(  \int_{0}^{t} \frac{d\tau}{(t-\tau)^{\frac{1+n/p}{\alpha}} \, \tau^{\frac{n}{\alpha p}}} \right) \| w\|^2_{E_T}\\
       \leq &\, C\, T^{-\frac{1}{\alpha}-\frac{n}{\alpha p}+1}\, \| w\|^{2}_{E_T}.
    \end{split}
\end{equation*}
As before, note that this integral converges due to the well-known fact that  $p>\frac{n}{\alpha-1} >\frac{n}{\alpha}$. 

\medskip

Gathering  these estimates, we obtain
\begin{equation}\label{Estim-Nonlinear}
  \left\| \int_{0}^{t} p_\alpha(t-\tau,\cdot)\ast \text{div}(w\, {\bf A}(w))(\tau,\cdot)\, d \tau \right\|_{E_T}\leq  C\, T^{-\frac{1}{\alpha}-\frac{n}{\alpha p}+1}\, \| w\|^{2}_{E_T}.
\end{equation}
Therefore, since $w\in B_R$, and using the estimate (\ref{Estim-Data-Terms}), we can write
\begin{equation*}
\begin{split}
   C\, T^{-\frac{1}{\alpha}-\frac{n}{\alpha p}+1}\, \| w\|^{2}_{E_T}
  \leq &\, C\, T^{-\frac{1}{\alpha}-\frac{n}{\alpha p}+1}\, \| w\|^{2}_{E_T} \leq \, C\, T^{-\frac{1}{\alpha}-\frac{n}{\alpha p}+1}\, \big(\| w-U_0\|_{E_T}+ \|U_0\|_{E_T}\big)^{2}\\
  \leq &\, C\, T^{-\frac{1}{\alpha}-\frac{n}{\alpha p}+1}\, \big( R+ C(\| v_0\|_{L^p}+\|g\|_{L^p}) \big)^2. 
\end{split}
\end{equation*}
Hence, we fix the time $T$ sufficiently small such that 
\begin{equation}\label{Control-Time-1}
  T^{-\frac{1}{\alpha}-\frac{n}{\alpha p}+1}\, \big( R+ C(\| v_0\|_{L^p}+\|g\|_{L^p}) \big)^2 \leq R.  
\end{equation}
\end{proof}

\begin{Lemme}\label{Lem-Parabolic-2} Assume that the operator ${\bf A}(\cdot)$ verifies (\ref{Condition-A}). Fix the time $0<T\leq 1$ sufficiently small as in expression (\ref{Condition-Time-2}) below. Then, the operator $T_\alpha(\cdot): B_R \to B_R$ is contractive. 
\end{Lemme}
\begin{proof}  Let $w_1,w_2 \in B_R$. From the expression (\ref{Mild}) it follows that 
\begin{equation}\label{Equation-Difference}
    \mathcal{T}_\alpha (w_1)-\mathcal{T}_\alpha(w_2)=\int_{0}^{t}p_\alpha\ast \text{div}\big( w_1({\bf A}(w_1)-{\bf A}(w_2))+ (w_1-w_2)\, {\bf A}(w_2) \big)(\tau,\cdot)d\tau.
\end{equation}
Then, following the same estimates used to prove (\ref{Estim-Nonlinear}) and applying  the estimate (\ref{Estim-Data-Terms}), we find that
\begin{equation*}
\begin{split}
    &\, \|  \mathcal{T}_\alpha (w_1)-\mathcal{T}_\alpha(w_2)\|_{E_T} \leq C\, T^{-\frac{1}{\alpha}-\frac{n}{\alpha p}+1}\, \big(\| w_1\|_{E_T}+\|w_2\|_{E_T} \big)\| w_1 - w_2\|_{E_T}\\
    \leq &\, C\, T^{-\frac{1}{\alpha}-\frac{n}{\alpha p}+1}\, \big(\| w_1-U_0\|_{E_T}+\|w_2-U_0\|_{E_T}+2\|U_0\|_{E_T} \big)\| w_1 - w_2\|_{E_T}\\
    \leq &\,  C\, T^{-\frac{1}{\alpha}-\frac{n}{\alpha p}+1}\, \Big(2R+2C(\|v_0\|_{L^p}+\|g\|_{L^p})\Big)\| w_1 - w_2\|_{E_T}.
 \end{split}
\end{equation*}
Consequently, we fix the time $T$ such that 
\begin{equation}\label{Condition-Time-2}
    C\, T^{-\frac{1}{\alpha}-\frac{n}{\alpha p}+1}\, \Big(2R+2C(\|v_0\|_{L^p}+\|g\|_{L^p})\Big)<1.
\end{equation}
\end{proof}

Having proven Lemmas \ref{Lem-Parabolic-1} and \ref{Lem-Parabolic-2}, we fix the time $0<T=T_0\leq 1$ sufficiently small so that both conditions (\ref{Control-Time-1}) and (\ref{Condition-Time-2}) hold. Thus, we obtain a unique solution $v \in B_R \subset E_{T_0}\subset \mathbf{C}_{*}([0,T_0],L^p(\R^n))$ to equation (\ref{Mild}).

\medskip

Now, we prove the uniqueness of this solution in the largest space $\mathbf{C}_{*}\big([0,T_0],L^p(\R^n)\big)$. Assume that $v_1, v_2 \in \mathbf{C}_{*}\big([0,T_0],L^p(\R^n)\big)$ are two solutions to equation (\ref{Evolution-Equation}) with the same initial datum $v_0$ and external source $f$. Defining $w:=v_1-v_2$, we see  from expression (\ref{Equation-Difference}) that  this function  satisfies  the fixed-point problem:
\begin{equation*}
    w(t,\cdot)=\int_{0}^{t}p_\alpha(t-\tau,\cdot)\ast \text{div}\big( v_1 ({\bf A}(v_1)-{\bf A}(v_2)) - w\, {\bf A}(v_2)\big)(\tau,\cdot)d\tau. 
\end{equation*}
Then, for the parameter $q$ such that $1=1/q+1/p$ and using the relation $1+1/p=1/q+2/p$, we apply  Young and H\"older inequalities. Additionally, from the   the second point of Lemma \ref{Lem-Estim-Frac-Heat-Kernel} and the assumption (\ref{Condition-A}),  we obtain
\begin{equation*}
    \begin{split}
        \| w(t,\cdot)\|_{L^p}\leq &\, \int_{0}^{t} \|\vec{\nabla} p_\alpha(t-\tau,\cdot)\|_{L^q}\, \left\|  \big( v_1 ({\bf A}(v_1)-{\bf A}(v_2)) - w\, {\bf A}(v_2)\big)(\tau,\cdot)\right\|_{L^{\frac{p}{2}}}\, d\tau\\
        \leq &\,C\, \int_{0}^{t} \frac{1}{(t-\tau)^{\frac{1+n/p}{\alpha}}}\, \big( \| v_1(t,\cdot)\|_{L^p}\, \|({\bf A}(v_1)-{\bf A}(v_2))(\tau,\cdot)\|_{L^p}+ \|w(\tau,\cdot)\|_{L^p}\, \|{\bf A}(v_2)(\tau,\cdot)\|_{L^p}\big)d\tau\\
        \leq &\, C\, \int_{0}^{t} \frac{1}{(t-\tau)^{\frac{1+n/p}{\alpha}}}\, \big( \| v_1(\tau,\cdot)\|_{L^p}+\|v_2(\tau,\cdot)\|_{L^p}\big)\| w(\tau,\cdot)\|_{L^p}\,d\tau\\
        \leq &\, C\, t^{-\frac{1}{\alpha}-\frac{n}{\alpha p}+1}\, \big( \|v_1\|_{L^\infty_t L^p_x}+\|v_2\|_{L^\infty_t L^p_x}\big)\| w\|_{L^\infty_t L^p_x}.
    \end{split}
\end{equation*}
Consequently, by well-known arguments, the uniqueness of the solution follows from this estimate. Proposition \ref{Prop-Evolution-Lp} is now proven.  \end{proof}

\subsection*{Step 2: Global boundedness of the solution $u$} Proposition \ref{Prop-Evolution-Lp} is the key tool for proving the following:
\begin{Proposition}\label{Prop-Global-Boundedness} Assume that the operator ${\bf A}(\cdot)$ verifies (\ref{Condition-A}). Let $\alpha>1$,   $\max\left(1, \frac{n}{\alpha-1} \right)<p<+\infty$ and let $u \in L^p(\R^n)$ be a  solution of equation (\ref{Equation}).  

\medskip

For any $1<r<+\infty$ and $s\geq 0$, assume that the external source  verifies $f\in \dot{W}^{-1,r}\cap \dot{W}^{s,r}(\R^n)$. Then, it holds $u\in L^\infty(\R^n)$. 
\end{Proposition}
\begin{proof} From the assumption above on $f$, using interpolation inequalities and setting $r=p$,  it follows  that $f\in L^p(\R^n)$. Then, in the setting of the evolution problem (\ref{Evolution-Equation}), we fix the initial datum $v_0=u$, the external source $g=f$, and  denote by $v\in \mathcal{C}_{*}\big([0,T_0], L^p(\R^n)\big)$ the unique associated solution obtained from Proposition \ref{Prop-Evolution-Lp}. Additionally, recall that this solution verifies the estimate (\ref{Estim-Solution-Besov}).

\medskip

On the other hand, since the solution $u \in L^p(\R^n)$ of equation (\ref{Equation}) is a time-independent function (in particular, $\partial_t u =0$), it is also a solution of the evolution problem (\ref{Evolution-Equation}) with the same initial datum $v_0=u$ and the external source $f$. Moreover, it directly  follows that  $u \in \mathcal{C}_{*}([0,T_0],L^p(\R^n))$.

\medskip

Finally, from the uniqueness of solutions in the space $\mathcal{C}_{*}([0,T_0],L^p(\R^n))$, the  identity $v=u$ holds . Consequently, the solution $u$ also verifies the estimate (\ref{Estim-Solution-Besov}), which implies that $u\in L^\infty(\R^n)$. 
\end{proof}

\subsection*{Step 3: Gain of fractional regularity for the solution $u$} From  Proposition \ref{Prop-Global-Boundedness} it holds that $u \in L^\infty(\R^n)$. Additionally, from the assumption in Theorem \ref{Th-Regularity} it also holds that $u\in L^{1,\infty}(\R^n)$. Consequently, by the first point of Lemma \ref{Lem-Interpolation} we obtain that $u \in L^r(\R^n)$ for any $1<r<+\infty$. This fact is a  key tool for proving the following:
\begin{Proposition}\label{Prop-Gain-Regularity} Under the same hypotheses as in Proposition \ref{Prop-Global-Boundedness}, assume additionally  that the operator ${\bf A}(\cdot)$ satisfies  (\ref{Condition-A-Sobolev}) and the solution also satisfies $u\in L^{1,\infty}(\R^n)$.  Then, for any $1<r<+\infty$, we have $u \in \dot{W}^{s+\alpha,p}(\R^n)$.
\end{Proposition}
\begin{proof} We begin  by explaining the  strategy of the proof. For given $\alpha>1$ and $s\geq 0$,  we consider the quantities $s+\alpha>1$ and $\alpha-1>0$. Consequently, we can find a parameter  $k \in \mathbb{N}$  such that $ k(\alpha-1) \leq s+\alpha \leq (k+1)(\alpha-1)$. Then, for $0 \leq \varepsilon < \alpha-1$,  we can write 
\begin{equation}\label{Decomposition-s-alpha}
s+\alpha= k(\alpha-1)+\varepsilon.    
\end{equation}
In this context,  in order to prove the gain of regularity $u \in \dot{W}^{s+\alpha,r}(\R^n)$, we first show that $u \in \dot{W}^{k(\alpha-1),r}(\R^n)$,  and then verify that $(-\Delta)^{\frac{k(\alpha-1)}{2}}u \in \dot{W}^{\varepsilon,r}(\R^n)$.    

\medskip

To show that $u \in \dot{W}^{k(\alpha-1),r}(\R^n)$, we will apply an iteration process with respect to the parameter $k$.
 \begin{itemize}
    \item The case $k=0$. Since $u\in L^{1,\infty}\cap L^\infty(\R^n)$, from the first point of Lemma \ref{Lem-Interpolation} it follows that  $u \in L^r(\R^n)$, for any $1<r<+\infty$. 

    \medskip

    \item The case $k=1$.  Since  $u$ solves  fixed-point equation (\ref{Fixed-point}), we can write 
\begin{equation}\label{Identity-Reg-1}
(-\Delta)^{\frac{\alpha-1}{2}} u = -(-\Delta)^{-\frac{1}{2}} \text{div}\big( u\, {\bf A}(u)\big)+(-\Delta)^{-\frac{1}{2}} f,     
\end{equation}
where  we will prove that each term on right-hand side belong to the space $L^r(\R^n)$.

\medskip

For the first term,  since $u \in L^r(\R^n)$ and  from the assumption (\ref{Condition-A-Sobolev}) on the operator ${\bf A}(\cdot)$, we first obtain that ${\bf A}(u)\in L^r(\R^n)$. Then, from the additional information $u\in L^\infty(\R^n)$, it follows that  $(-\Delta)^{-\frac{1}{2}} \text{div}\big( u\, {\bf A}(u)\big) \in L^r(\R^n)$. For the second term, from the  assumption on $f$ given in the first point of Theorem \ref{Th-Regularity},  we directly have  $(-\Delta)^{-\frac{1}{2}} f \in L^r(\R^n)$. 

 \medskip
 
 Consequently, for any $1<r<+\infty$, we conclude that  $u \in \dot{W}^{\alpha-1,r}(\R^n)$.

 \medskip

 \item The case $k\geq 2$.  Applying the operator $(-\Delta)^{\frac{\alpha-1}{2}}$ to the identity (\ref{Identity-Reg-1}), we obtain 
\begin{equation*}
(-\Delta)^{\frac{2(\alpha-1)}{2}} u = -(-\Delta)^{\frac{\alpha-1}{2}} (-\Delta)^{-\frac{1}{2}} \text{div}\big( u\, {\bf A}(u)\big) + (-\Delta)^{\frac{\alpha-1}{2}} (-\Delta)^{-\frac{1}{2}} f.     
\end{equation*}
As before, we must  prove that each term on the right-hand side  belong to the space $L^r(\R^n)$. 

\medskip 

For the first term, we write 
\[ (-\Delta)^{\frac{\alpha-1}{2}} (-\Delta)^{-\frac{1}{2}} \text{div}\big( u\, {\bf A}(u)\big)= (-\Delta)^{-\frac{1}{2}} \text{div}\big((-\Delta)^{\frac{\alpha-1}{2}} ( u\, {\bf A}(u))\big). \]
Here, for $1<r_1,r_2<+\infty$ such that $\frac{1}{r}=\frac{1}{r_1}+\frac{1}{r_2}$, from  Lemma  \ref{Leibniz-rule}  it follows that
\begin{equation}\label{Estim-Leibniz-Rule}
    \left\|(-\Delta)^{\frac{\alpha-1}{2}} ( u\, {\bf A}(u)) \right\|_{L^r}\leq C\, \| (-\Delta)^{\frac{\alpha-1}{2}} u\|_{L^{r_1}}\, \| {\bf A}(u)\|_{L^{r_2}}+ C\,\| u\|_{L^\infty}\, \| (-\Delta)^{\frac{\alpha-1}{2}}{\bf A}(u)\|_{L^r}. 
\end{equation}
Recall that  from the previous case (when $k=1$) we have $u \in \dot{W}^{\alpha-1,r}(\R^n)$ for any $1<r<+\infty$. This fact, along with the assumption (\ref{Condition-A-Sobolev}) on the operator ${\bf A}(\cdot)$ implies that $\| (-\Delta)^{\frac{\alpha-1}{2}} u\|_{L^{r_1}}<+\infty$ and $\| (-\Delta)^{\frac{\alpha-1}{2}}{\bf A}(u)\|_{L^r}<+\infty$, respectively. Additionally, since  for any $1<r<+\infty$  we have $u \in L^r(\R^n)$,  from the assumption (\ref{Condition-A}) it follows that $\| {\bf A} (u) \|_{L^{r_2}}<+\infty$. Finally, recall that we also have $\| u \|_{L^\infty}<+\infty$. Consequently, each term on the right-hand side in the estimate above is bounded, hence  $\ds{(-\Delta)^{\frac{\alpha-1}{2}} ( u\, {\bf A}(u))  \in L^r(\R^n)}$.

\medskip

For the second term, we write  
\[ (-\Delta)^{\frac{\alpha-1}{2}} (-\Delta)^{-\frac{1}{2}} f=(-\Delta)^{\frac{2(\alpha-1)-\alpha}{2}} f, \]
where, since $s+\alpha=k(\alpha-1)+\varepsilon$, it holds that $2(\alpha-1)-\alpha \leq k(\alpha-1)-\alpha \leq s$. Consequently, using again the assumption on $f$ given in the first point of Theorem \ref{Th-Regularity},   we obtain $(-\Delta)^{\frac{2(\alpha-1)-\alpha}{2}} f \in L^r(\R^n)$. 

\medskip

Hence, for any $1<r<+\infty$, it follows that    $u \in \dot{W}^{2(\alpha-1),r}(\R^n)$.  By iterating this process until $k$,  we conclude that $u \in \dot{W}^{k(\alpha-1),r}(\R^n)$.   
\end{itemize}

\medskip

It remains  to prove that  $(-\Delta)^{\frac{\varepsilon}{2}}  u \in \dot{W}^{\frac{k(\alpha-1)}{2},r}(\R^n)$. To achieve this, from the identity (\ref{Fixed-point}) we write  
\begin{equation*}
\begin{split}
 (-\Delta)^{\frac{\varepsilon+k(\alpha-1)}{2}} u = &\, -(-\Delta)^{\frac{\varepsilon+k(\alpha-1)}{2}} (-\Delta)^{-\frac{\alpha}{2}} \text{div}(u\, {\bf A}(u))+ (-\Delta)^{\frac{\varepsilon+k(\alpha-1)}{2}}(-\Delta)^{-\frac{\alpha}{2}}f \\
=&\, - (-\Delta)^{-\frac{1}{2}}\text{div}\left( (-\Delta)^{\frac{\varepsilon+(k-1)(\alpha-1)}{2}}\big( u\, {\bf A}(u)\big)\right) + (-\Delta)^{\frac{\varepsilon+(k-1)(\alpha-1) -1}{2}}f. 
\end{split}
\end{equation*}

For the first term on the right-hand side, since $(-\Delta)^{\frac{k(\alpha-1)}{2}} u \in L^r(\R^n)$, $u\in L^r(\R^n)$ and $u \in L^\infty(\R^n)$, using again  Lemma \ref{Leibniz-rule}, we obtain   $(-\Delta)^{\frac{k(\alpha-1)}{2}} \big( u\, {\bf A}(u)\big) \in L^r(\R^n)$. Additionally,  since $(-\Delta)^{\frac{(k-1)(\alpha-1)}{2}} \big( u\, {\bf A}(u)\big) \in L^r(\R^n)$, where $0 \leq  \varepsilon < \alpha-1$,   using interpolation inequalities it follows that    $(-\Delta)^{\frac{\varepsilon+(k-1)(\alpha-1)}{2}} \big( u\, {\bf A}(u)\big)\in L^r(\R^n)$. Consequently, it holds  that $- (-\Delta)^{-\frac{1}{2}}\text{div}\left( (-\Delta)^{\frac{\varepsilon+(k-1)(\alpha-1)}{2}}\big( u\, {\bf A}(u)\big)\right) \in L^r(\R^n)$.

\medskip

For the second term on the right-hand side, since  $s+\alpha= k(\alpha-1)+\varepsilon$, it follows that 
\[\varepsilon+(k-1)(\alpha-1) -1=s. \]
Then, using the assumptions on $f$  it   directly follows that  $(-\Delta)^{\frac{\varepsilon+(k-1)(\alpha-1) -1}{2}}f \in L^p(\R^n)$. 

\medskip

We finally conclude that  $u \in \dot{W}^{s+\alpha,r}(\R^n)$ for any $1<r<+\infty$. 
\end{proof}

\subsection*{Step 4: Optimality of the gain in regularity}
\begin{Proposition} Under the same hypothesis as in Propositions \ref{Prop-Global-Boundedness} and \ref{Prop-Gain-Regularity},  assume that for $\varepsilon>0$, we have $f\notin \dot{W}^{s+\varepsilon,r}(\R^n)$. Then, it follows that $u\notin \dot{W}^{s+\alpha+\varepsilon,r}(\R^n)$. 
\end{Proposition}
\begin{proof} Assume, for contradiction, that $u \in \dot{W}^{s+\alpha+\varepsilon,r}(\R^n)$. Then, applying the operator $(-\Delta)^{\frac{s+\alpha+\varepsilon}{2}}$ in each term of   equation  (\ref{Fixed-point}), we obtain
\begin{equation}\label{Identity-Contradiction}
    (-\Delta)^{\frac{s+\alpha+\varepsilon}{2}} u + (-\Delta)^{\frac{s+\alpha+\varepsilon}{2}} (-\Delta)^{-\frac{\alpha}{2}}\text{div}(u\,{\bf A}(u))= (-\Delta)^{\frac{s+\varepsilon}{2}}f.
\end{equation}

We now rewrite the second term on the left-hand side as 
\begin{equation*}
\begin{split}
 (-\Delta)^{\frac{s+\alpha+\varepsilon}{2}} (-\Delta)^{-\frac{\alpha}{2}}\text{div}(u\,{\bf A}(u)) = &\, (-\Delta)^{\frac{s+1+\varepsilon}{2}} (-\Delta)^{-\frac{1}{2}}  \text{div}(u\,{\bf A}(u)) \\
 = &\,    (-\Delta)^{-\frac{1}{2}}  \text{div} \big( (-\Delta)^{\frac{s+1+\varepsilon}{2}} (u\,{\bf A}(u))  \big).   
\end{split}
 \end{equation*}
Then, since  $u \in \dot{W}^{s+\alpha+\varepsilon,r}(\R^n)$ (where $\alpha>1$), and given that $u\in L^r$ and $u\in L^\infty$, following the same arguments used in the right-hand side of estimate (\ref{Estim-Leibniz-Rule}), we obtain that  $(-\Delta)^{\frac{s+1+\varepsilon}{2}} (u\,{\bf A}(u)) \in L^r(\R^n)$.

\medskip

Consequently, each term on left hand side in identity (\ref{Identity-Contradiction}) belongs to the space $L^r(\R^n)$ implying  that $f \in \dot{W}^{s+\varepsilon,r}(\R^n)$. This contradicts the assumption that  $f \notin \dot{W}^{s+\varepsilon,r}(\R^n)$,  completing the proof.
\end{proof}
The proof of Theorem \ref{Th-Regularity} is now complete. 
\subsection{Proof of Corollary \ref{Corollary-Regularity}}  The proof is straightforward. Since for $\alpha>1$ and any $1<r<+\infty$, we have $u\in  L^r \cap \dot{W}^{s+\alpha,r}(\R^n)$, for any $k \in \mathbb{N}$ such that $k\leq [s]+1$, it follows from interpolation inequalities that $u\in \dot{W}^{k,r}(\R^n)$. In particular, for any multi-index $|{\bf a}|\leq k$, and using the continuous embedding $L^r(\R^n)\subset \dot{M}^{1,r}(\R^n)$, we obtain $\partial^{{\bf a}}u \in \dot{M^{1,r}}(\R^n)$.

\medskip

With this information at our disposal, in the setting of  Lemma \ref{Lem-Holder-Morrey}, we fix $r>n$ to directly obtain that $u \in \mathcal{C}^{[s],\sigma}(\R^n)$, where $\sigma=1-n/r$. 

\medskip

In the homogeneous case, when $f\equiv 0$, from Theorem \ref{Th-Regularity} it follows that $u \in \dot{W}^{s,r}(\R^n)$ for any $s\geq 0$. Therefore, using the same argument as above, we conclude that $u\in \mathcal{C}^{\infty}(\R^n)$.

\section{Fractional gain of regularity for the toy model}\label{Appendix} 
We consider the toy model for equation (\ref{Equation}), which is  given in equation (\ref{Toy-Model}). As explained in the introduction, the aim of this model is to show that the weaker nonlinear effects of the term $(-\Delta)^{\frac{\beta}{2}}(u^2)$ allow us to prove a gain in regularity for weak $L^p$-solutions to equation (\ref{Toy-Model}) in the case when $0<\alpha\leq 1$.

\medskip

In the following result, in order to focus directly on this gain in regularity, we assume the existence and global boundedness of these weak $L^p$-solutions. However, these facts could be proven by following arguments similar to those in the proofs of Propositions \ref{Th-Existence}, \ref{Prop-Evolution-Lp}, and \ref{Prop-Global-Boundedness}. 

\medskip

Additionally, we rewrite equation (\ref{Toy-Model}) as the following fixed-point problem:
\begin{equation}\label{Fixed-Point-Toy-Model}
   u= - (-\Delta)^{\frac{-\alpha+\beta}{2}}(u^2)+(-\Delta)^{-\frac{\alpha}{2}}f.  
\end{equation}
\begin{Proposition} Let $0<\beta<\alpha\leq 1$. For $1<p<+\infty$, let $u \in L^p\cap L^\infty(\R^n)$ be a weak solution to equation (\ref{Fixed-Point-Toy-Model}) associated with an external source $f$.

\medskip

For $s\geq 0$, assume that $f \in \dot{W}^{-\beta, p}\cap \dot{W}^{s,p}(\R^n)$. Then, it follows that $u \in \dot{W}^{s+\alpha,p}(\R^n)$. 
\end{Proposition}
\begin{proof} The proof essentially follows the same steps as those in the proof of Proposition \ref{Prop-Gain-Regularity}. Thus, we will only outline the main arguments.

\medskip

Since $\alpha-\beta>0$, following the idea in expression (\ref{Decomposition-s-alpha}), for $k\in \mathbb{N}$ and $0<\varepsilon<\alpha-\beta$, we can write
\[ s+\alpha= k(\alpha-\beta)+\varepsilon. \]
Consequently, we will first prove that $u \in \dot{W}^{k(\alpha-\beta),p}(\R^n)$ and then verify that $(-\Delta)^{\frac{\varepsilon}{2}}u \in \dot{W}^{k(\alpha-\beta),p}(\R^n)$.

\medskip

To prove the first claim, we always  proceed by induction  $k$. For $k=1$, applying the operator $(\Delta)^{\frac{\alpha-\beta}{2}}$ to equation (\ref{Fixed-Point-Toy-Model}), we obtain 
\begin{equation*}
    (-\Delta)^{\frac{\alpha-\beta}{2}}u= -u^2+(-\Delta)^{-\frac{\beta}{2}}f. 
\end{equation*}
Here, we can easily verify that each term on the right-hand side belongs to $L^p(\R^n)$. Therefore, we conclude that $u\in \dot{W}^{\alpha-\beta,p}(\R^n)$.

\medskip

Now, for $k\geq 2$, applying   the operator $(-\Delta)^{\frac{\alpha-\beta}{2}}$ once more to each term  in the previous identity, we obtain
\begin{equation*}
  (-\Delta)^{\frac{2(\alpha-\beta)}{2}}u= -(-\Delta)^{\frac{\alpha-\beta}{2}}(u^2)+(-\Delta)^{\frac{\alpha-\beta}{2}}(-\Delta)^{-\frac{\beta}{2}}f.  
\end{equation*}
For the first term on the right-hand side,  Lemma \ref{Leibniz-rule} yields  that 
\[\|  (-\Delta)^{\frac{\alpha-\beta}{2}}(u^2)\|_{L^p}\leq C\, \|  (-\Delta)^{\frac{\alpha-\beta}{2}} u\|_{L^p}\, \| u \|_{L^\infty}<+\infty. \]
The second term on the right-hand side can be estimated directly using the assumptions on the external source $f$.

\medskip

Thus, we conclude that $u\in \dot{W}^{2(\alpha-\beta),p}(\R^n)$, and by iterating this process up to $k$, it follows that $u\in \dot{W}^{k(\alpha-\beta),p}(\R^n)$. Finally, the fact that $(-\Delta)^{\frac{\varepsilon}{2}}u \in \dot{W}^{k(\alpha-\beta),p}(\R^n)$ follows by similar arguments.
\end{proof}
\begin{Remarque} Note that, in contrast to equation (\ref{Equation}), the singular integral operator ${\bf A}(\cdot)$ is not included in the simpler model (\ref{Toy-Model}). Therefore, the additional (technical) assumption that $u \in L^{1,\infty}(\R^n)$ is no longer required in this case.
\end{Remarque}

\paragraph{{\bf Statements and Declaration}}
Data sharing does not apply to this article, as no datasets were generated or analyzed during the current study. In addition, the author declares that he has no conflicts of interest and that there are no funding supports to declare.

\vspace{1cm}

\end{document}